\documentclass{amsart}

\usepackage{amssymb}

\newtheorem{thm}{Theorem}[section] 

\newtheorem{cor}[thm]{Corollary}

\newtheorem{lem}[thm]{Lemma}
\newtheorem{prop}[thm]{Proposition}

\newtheorem{rem}[thm]{Remark}

\begin{document}

\title{How do algebras grow?}

\author{Be'eri Greenfeld}

\address{Department of Mathematics, Bar Ilan University, Ramat Gan 5290002, Israel}
\email{beeri.greenfeld@gmail.com}

\begin{abstract}
We construct an increasing, submultiplicative, arbitrarily rapid function $f:\mathbb{N}\rightarrow \mathbb{N}$ which is not equivalent to the growth function of any finitely generated algebra, demonstrating the difficulty in characterizing growth functions in an asymptotic language.
\end{abstract}

\maketitle

\section{Introduction}

\subsection{Growth functions} The question of `how do algebras grow?', or, which functions can be realized as growth functions of algebras (associative/Lie/Jordan/other, or algebras having certain additional algebraic properties) is a major problem in the meeting point of several mathematical fields including algebra, combinatorics, symbolic dynamics and more.

In this note we examine growth functions of infinite dimensional, finitely generated associative algebras. Let $F$ be an arbitrary field and let $R$ be such $F$-algebra. Fixing a finite dimensional generating subspace $R=F[V]$ we define the growth of $R$ to be the function:
$$\gamma_{R,V}(n)=\dim_F V^n$$
This evidently depends on the choice of $V$, but might change only up to the following equivalence relation: $f\sim g$ if $f(n)\leq g(Cn)\leq f(Dn)$ for some $C,D>0$. Therefore when talking about the growth of an algebra we refer to the $\sim$-equivalence class of the function $\gamma_{R,V}(n)$ (for some $V$). We say that $f\preceq g$ if $f(n)\leq g(Cn)$ for some $C>0$. For more on growth functions of algebras, see \cite{KrauseLenagan}.

There are obvious properties necessarily satisfied by such growth functions; they are always:

\begin{itemize}
    \item \textit{Increasing} (namely, $f(n)<f(n+1)$); and
    \item \textit{Submultiplicative} (namely, $f(n+m)\leq f(n)f(m)$).
\end{itemize}  
The main goal in studying the variety of possible growth functions is to investigate to what extent these conditions are in fact sufficient. 

\subsection{Former results} Several attempts have been made to realize as wide as possible variety of such functions as growth functions of associative algebras.

Smoktunowicz and Bartholdi \cite{SmoktunowiczBartholdi} proved that every increasing and submultiplicative function is equivalent to a growth function of an associative algebra, up to a polynomial factor. Namely:
\begin{thm}[{\cite[Theorem~C]{SmoktunowiczBartholdi}}] 
Let $f:\mathbb{N}\rightarrow \mathbb{N}$ be submultiplicative and increasing. Then there exists a finitely generated monomial algebra $B$ whose growth function satisfies: $$f(n)\preceq \dim_F B(n)\preceq n^2f(n).$$
\end{thm}
They deduce the following corollary which allows an accurate realization of `sufficiently regular' rapid growth types:
\begin{cor}[{\cite[Corollary~D]{SmoktunowiczBartholdi}}]
Let $f:\mathbb{N}\rightarrow \mathbb{N}$ be a submultiplicative, increasing, and such that $f(Cn)\geq nf(n)$ for some $C>0$ and all $n\in \mathbb{N}$. Then there exists an associative algebra of growth $\sim f$.
\end{cor}

It should be mentioned that the above constructions were modified by the author in \cite{BGJalg, BGIsr} to construct prime, primitive and simple algebra of prescribed growth rates; these also yield the existence of finitely generated simple Lie algebras with arbitrary growth functions satisfying the conditions of the above corollary (arguments will appear elsewhere).

Bell and Zelmanov \cite{BellZelmanov} found an additional condition (on discrete derivatives) satisfied by all growth functions; their remarkable achievement is that in fact, \textit{every increasing function satisfying this condition is equivalent to a growth function of an associative algebra}. They proved:
\begin{thm}[{\cite[Theorem~1.1]{BellZelmanov}}] 
A growth function of an algebra is asymptotically equivalent to a constant function, a linear function, or a weakly increasing function $F:\mathbb{N}\rightarrow\mathbb{N}$ with the following properties:
\begin{enumerate}
    \item $F'(n)\geq n+1$ for all $n\in \mathbb{N}$;
    \item $F'(m)\leq F'(n)^2$ for all $m\in \{n,\dots,2n\}$.
\end{enumerate}
Conversely, if $F(n)$ is either a constant function, a linear function, or a weakly increasing function
with the above properties then it is asymptotically equivalent to the growth function of an finitely generated algebra.
\end{thm}
As the writers suggest, one can interpret this theorem as saying that other than the necessary condition that $F'(m)\leq F'(n)^2$ for all $m\in \{n,\dots,2n\}$, which is related to submultiplicativity, the only additional
constraints required for being realizable as a growth function of an algebra are those coming from Bergman's gap theorem \cite{BergmanGap} (which asserts that a super-linear growth function must be at least quadratic) and the elementary "gap" that an algebra cannot have strictly sublinear growth that is not constant. However, it seems that there is no natural characterization of whether a given function is equivalent to a function satisfying the above condition on discrete derivatives.

We remark that there exist extremely pathological examples of oscillating growth of algebras: Trofimov \cite{Semigroup} showed that for every $f_{-}(n)\succ n^2$ and $f_{+}\prec \exp(n)$ there exists a $2$-generated semigroup with growth function infinitely often smaller than $f_{-}$ and infinitely often larger than $f_{+}$. This was improved by Belov, Borisenko and Latyshev in \cite{BBL1997}; such examples cannot be found within the class of groups.

\subsection{Our aim} In this note we construct an example emphasizing the difficulty of the fundamental questions of characterizing growth functions of algebras.
Namely, we prove:
\begin{thm}\label{main thm submul}
Let $g:\mathbb{N}\rightarrow\mathbb{N}$ be a subexponential function. Then there exists an increasing, submultiplicative function $f:\mathbb{N}\rightarrow \mathbb{N}$ such that $f\succeq g$ and $f$ is not equivalent to the growth function of any finitely generated algebra.
\end{thm}

We mention that this in particularly implies that $f$ constructed above is also not equivalent to the growth function of any group; Bartholdi and Erschler \cite{BartholdiErschler} proved that any function $f:\mathbb{N}\rightarrow \mathbb{N}$ which grows uniformly faster\footnote{In the sense that $f(2n)\leq f(n)^2\leq f(\eta n)$ for $n\gg 1$.} than $\exp(n^\alpha)$ are equivalent to growth functions of groups ($\alpha=\log 2/\log \eta\approx 0.7674$ where $\eta$ is the positive root of $X^3-X^2-2X-4$). They also leave open the question of providing a complete characterization of growth functions of groups.

By Bergman's gap theorem, the function $n^{\alpha}$ for $\alpha\in(1,2)$, which is increasing and submultiplicative, is not equivalent to the growth function of an algebra; no similar gap theorem is valid for other polynomially bounded functions. Theorem \ref{main thm submul} shows that the two necessary conditions of being increasing and sumbultiplicative are \textit{not sufficient even for `sufficiently rapid' functions}, thereby emphasizing the significance of the polynomial factor in \cite[Theorem~C]{SmoktunowiczBartholdi}. This phenomenon hints that there is in fact no characterization in an asymptotic language of growth functions within the class of increasing and submultiplicative functions; this might justify and emphasize the importance of using new characteristics of functions in the attempt to characterize growth functions, such as discrete derivatives as done in \cite{BellZelmanov}.

\section{Preliminary results}

By a result of Bell and Zelmanov \cite[Proposition~2.1]{BellZelmanov}, if $\gamma$ is a growth function of an finitely generated algebra then $\gamma'(m)\leq \gamma'(n)^2$ for every $m\in\{n,\dots,2n\}$, where $\gamma'(n) = \gamma(n)-\gamma(n-1)$. Their proof yields:
\begin{rem}
Assume $\gamma$ is a growth function of an algebra. Let $d \in \mathbb{N}$. Then 
$\gamma'(m)\leq \gamma'(n)^d$ for every $m\in\{n,\dots,dn\}$.
\end{rem}
\begin{proof}
We may assume the algebra is monomial, so $\gamma'(n)$ is the number of (nonzero) words of length $n$ in the generators.
But if $n \leq m \leq dn$ then every word of length $m$ is a prefix of a product of $d$ words of length $n$.
\end{proof}

This is used in the next proposition.

\begin{prop} \label{property_growth}
Suppose $f:\mathbb{N}\rightarrow \mathbb{N}$ is equivalent to a growth function $\gamma:\mathbb{N}\rightarrow \mathbb{N}$ of an finitely generated algebra. Then there exists $C\in \mathbb{N}$ such that for all $D\gg 1$, for all $n$ we have: $$f(2CDn)-f(2CDn-C)\leq 2D^2n (f(CDn)-f(Cn-C))^{2D}.$$
\end{prop}

\begin{proof}
Write $\gamma(n)\leq f(Cn)\leq \gamma(Dn)$ for some $C,D$ (we can take $D$ arbitrarily large).
Set $h(n)=f(Cn)$ and $\varphi(n)=\gamma(Dn)$. Then $h(n)\leq \varphi(n)\leq h(Dn)$.
Observe that: $$h'(n)=f(Cn)-f(Cn-C)\leq \gamma(Dn)-\gamma(n-1)=\sum_{k=n}^{Dn} \gamma'(k)\leq Dn\gamma'(n)^D.$$
Note also that: 
\begin{eqnarray*}
\gamma'(Dn) & =& \gamma(Dn)-\gamma(Dn-1) \\ 
& \leq & \gamma(Dn)-\gamma(Dn-D) \\
& = & \varphi(n)-\varphi(n-1) \\
& \leq & h(Dn)-h(n-1).
\end{eqnarray*}
Putting these together, we get that:
\begin{eqnarray*}
f(2CDn)-f(2CDn-C) & = & h'(2Dn) \\
& \leq & 2D^2n\gamma'(2Dn)^D 
\\
& \leq & 2D^2n\gamma'(Dn)^{2D} \\ 
& \leq & 2D^2n (h(Dn)-h(n-1))^{2D} \\ & = & 2D^2n (f(CDn)-f(Cn-C))^{2D},
\end{eqnarray*}
as desired.
\end{proof}

\section{A construction of a submultiplicative function}\label{construction of f}

Let $1 < d_1<d_2<\cdots$ be an increasing sequence, and $n_1,n_2,\dots$ a sequence such that 
$$n_1 < d_1 n_1 < n_2 < d_2 n_2 < n_3 < \cdots.$$ 
Both sequences are to be restricted in the sequel
by conditions of the form ``$d_k$ is greater than a function of $\{d_i,n_i\}_{i=0}^{k-1}$'' and ``$n_k$ is greater than a function of $\{d_i,n_i\}_{i=0}^{k-1}$ and $d_k$''.

\subsection{The interval $[1,n_2]$}
We will define a function $f:\mathbb{N}\rightarrow \mathbb{N}$, first by defining it on the domain $[1,n_2]$:
\begin{itemize}
\item For $x\leq n_1$, take $f(x)=2^x$;
\item For $n_1<x\leq d_1n_1$, take $f(x)=f(x-1)+x+1$;
\item For $d_1n_1<x\leq n_2$, take $f(x)=\lfloor 2^{1/{2d_1}}f(x-1)\rfloor$.
\end{itemize}

Denote $\alpha_1=f(d_1n_1)-f(n_1)
< d_1^2n_1^2$. Since $\alpha_1$ is polynomial with respect to $n_1$ (assuming $d_1$ was fixed), if we take $n_1\gg 1$ then we may assume that $f(d_1n_1) = 2^{n_1}+\alpha_1\leq 2^{n_1+\frac{1}{3}}$.
We will also need the following fact:

\begin{lem}\label{floor}
Given $c>1$ and $\varepsilon>0$, for all $a_0\gg 1$ the sequence $a_{k+1}=\lfloor ca_k \rfloor$ satisfies $c^{k-\varepsilon}a_0\leq a_k\leq c^ka_0$.
\end{lem}

\begin{proof}
By induction $a_k\geq c^ka_0-\frac{c^k-1}{c-1}$, so: 
$$a_k-c^{k-\varepsilon}a_0\geq (c^k-c^{k-\varepsilon})a_0-\frac{c^k-1}{c-1}\xrightarrow{a_0\rightarrow \infty}\infty.$$
\end{proof}

Using Lemma \ref{floor} (taking $c = 2^{\frac{1}{2d_1}}$, $\varepsilon = 2^{-3}$), we can take $n_1\gg 1$ so that if $x\geq d_1n_1$ then $f(x)\geq f(d_1n_1)\cdot 2^{\frac{x-d_1n_1}{2d_1}-2^{-3}}$.
It is evident that $f$ is increasing in $[1,d_1n_1]$; it is also increasing in $[d_1n_1,n_2]$ if we only make sure $n_1$ is large enough. 
We now turn to prove that $f$ is submultiplicative.

\begin{prop}\label{f submul 0}
The function $f:[1,n_2]\rightarrow \mathbb{N}$ constructed above is submultiplicative.
\end{prop}
\begin{proof}
We first take care of the interval $[d_1n_1+1,n_2]$. Pick $p+q>d_1n_1$ with $p\leq q$ and we must show that $f(p+q)\leq f(p)f(q)$. We assume $d_1>2$ and $n_1\gg 1$ (this will be explicitly explained) and compute that:
\begin{eqnarray*}
f(p+q)&\leq &f(d_1n_1)\cdot 2^{\frac{p+q-d_1n_1}{2d_1}}\\
&\leq & 2^{n_1+\frac{1}{3}+\frac{p+q-d_1n_1}{2d_1}}\\
&=& 2^{\frac{1}{2}n_1+\frac{p+q}{2d_1}+\frac{1}{3}}.
\end{eqnarray*}
Assume $p\leq n_1$, then $q>n_1$. Whether or not $q \leq d_1n_1$, we have that: $$\frac{f(p+q)}{f(q)}\leq 2^{\frac{p}{2d_1}}\leq 2^p=f(p);$$ this follows since the ratio between two successive numbers in $[n_1,n_2]$ is $\leq \frac{1}{2d_1}$ if we only take $n_1\gg 1$). Thus we suppose $n_1<p$ (so in particular $f(q)\geq f(p)\geq 2^{n_1}$).
\begin{itemize}
    \item If $d_1n_1\leq p$ then (assuming $n_1>1$): 
    \begin{eqnarray*}
    f(p)f(q)&\geq & 2^{n_1+\frac{p-d_1n_1}{2d_1}-2^{-3}}2^{n_1+\frac{q-d_1n_1}{2d_1}-2^{-3}}\\
    &=& 2^{n_1+\frac{p+q}{2d_1}-2^{-2}}\\ 
    &\geq & 2^{\frac{1}{2}n_1+\frac{p+q}{2d_1}+\frac{1}{3}}\geq f(p+q).
    \end{eqnarray*}
    \item If $q\leq d_1n_1$ then: $$f(p)f(q)\geq 2^{2n_1}\geq 2^{\frac{1}{2}n_1+n_1+\frac{1}{3}}\geq f(p+q),$$ the latter inequality follows since $p+q\leq 2d_1n_1$.
    \item If $p<d_1n_1<q$ then: \begin{eqnarray*}
    f(p)f(q)& \geq &  2^{n_1}2^{n_1+\frac{q-d_1n_1}{2d_1}-2^{-3}} \\ & =& 2^{\frac{3}{2}n_1+\frac{q}{2d_1}-2^{-3}} \\ & = &  2^{\frac{3}{2}n_1+\frac{p+q}{2d_1}-\frac{p}{2d_1}-2^{-3}}\\
    & \geq & 2^{\frac{3}{2}n_1+\frac{p+q}{2d_1}-\frac{1}{2}n_1-2^{-3}}>f(p+q).
    \end{eqnarray*}
\end{itemize}
As for submultiplicativity in the interval $[n_1,d_1n_1]$ (note that the interval $[1,n_1]$ is trivial as the function $x\mapsto 2^x$ is submultiplicative), assume $n_1\leq p+q\leq d_1n_1$.
\begin{itemize}
    \item If $p>n_1$ then: $$f(p)f(q)\geq f(n_1)^2=2^{2n_1}\geq 2^{n_1}+\alpha_1=f(d_1n_1)\geq f(p+q).$$
    \item If $q<n_1$ then: $$f(p)f(q)=2^{p+q}\geq f(p+q).$$
    \item If $p\leq n_1\leq q$ then:
    \begin{eqnarray*}
    f(p+q) & \leq & f(d_1n_1)=2^{n_1}+\alpha_1 \\
    &\leq & 2^{n_1+\frac{1}{3}}\leq 2^p2^{n_1}=f(p)f(n_1)\leq f(p)f(q).
    \end{eqnarray*}
\end{itemize}
We thus proved that $f:[1,n_2]\rightarrow \mathbb{N}$ is submultiplicative.
\end{proof}

\subsection{Extending $f$ to $\mathbb{N}$}
We now extend $f$ to $\mathbb{N}$ as follows. Suppose $d_1,\dots,d_{k-1},\\ n_1,\dots,n_{k-1}$ were chosen and suppose $f$ was defined in the domain $[1,n_k]$ (we choose $n_k$ only after $\{d_i,n_i\}_{i=0}^{k-1},d_k$ were fixed).
Assumptions on $d_k,n_k$ will be explicitly made during the proof of submultiplicativity, in order to clarify where these assumptions originate from.
We assume $d_k\geq \frac{n_{k-1}}{2d_1\cdots d_{k-2}}$.
Define:
\begin{itemize}
    \item For $n_k<x\leq d_kn_k$, take $f(x)=f(x-1)+x+1$;
    \item For $d_kn_k<x\leq n_{k+1}$, take $f(x)=\lfloor 2^{\frac{1}{2d_1\cdots d_k}}f(x-1) \rfloor$.
\end{itemize}
Note that by taking $n_k$ to be large enough we can make sure that $f(x)$ is increasing.

\textbf{Condition (I).} We pick $n_k$ large enough such that for all $d_kn_k\leq x \leq n_{k+1}$ we have that: $$f(x)\geq f(d_kn_k)\cdot 2^{\frac{x-d_kn_k}{2d_1\cdots d_k}-2^{-k-2}}$$ (this is possible by Lemma \ref{floor} applied with $c=2^{\frac{1}{2d_1\cdots d_k}}$ and $\varepsilon=2^{-k-2}$).

\begin{lem} \label{lower bound f}
We can choose $\{n_k\}_{k\geq 1}$ to be sufficiently large such that for all $x\leq d_kn_k$ we have that: $$f(x)\geq 2^{\frac{x}{2d_1\cdots d_k}+1+2^{-k-1}}.$$
\end{lem}
\begin{proof}
Now we prove the assertion by induction on $k$. For $k=1$, in the relevant interval (namely $x\leq d_1n_1$) we have that: $$f(x)\geq 2^{n_1}\geq 2^{\frac{d_1n_1}{2d_1}+1+2^{-2}}$$ so the assertion is true (indeed, we take $n_1\geq 3$). Suppose the claim holds for $x\leq d_kn_k$ and let us prove it for $x\leq d_{k+1}n_{k+1}$; if $x\leq d_kn_k$ this is immediate from the hypothesis. If $d_kn_k<x\leq n_{k+1}$ then by Condition (I): $$f(x)\geq f(d_kn_k)\cdot 2^{\frac{x-d_kn_k}{2d_1\cdots d_{k}}-2^{-k-2}}.$$ We can bound the latter term from below (using what we just proved for $d_kn_k$):
\begin{eqnarray*}
f(d_kn_k)\cdot 2^{\frac{x-d_kn_k}{2d_1\cdots d_{k}}-2^{-k-2}}& \geq & 2^{\frac{d_kn_k}{2d_1\cdots d_{k}}+1+2^{-k-1}}\cdot 2^{\frac{x-d_kn_k}{2d_1\cdots d_{k}}-2^{-k-2}}\\ 
&= & 2^{\frac{x}{2d_1\cdots d_{k}}+1+2^{-(k+1)-1}}.
\end{eqnarray*}
If $n_{k+1}<x\leq d_{k+1}n_{k+1}$ then (using what we already know for $n_{k+1}$): 
\begin{eqnarray*}
f(x)\geq f(n_{k+1})& \geq & 2^{\frac{n_{k+1}}{2d_1\cdots d_{k}}+1+2^{-k-2}}\\
&=& 2^{\frac{d_{k+1}n_{k+1}}{2d_1\cdots d_{k+1}}+1+2^{-(k+1)-1}}\\
&\geq & 2^{\frac{x}{2d_1\cdots d_{k+1}}+1+2^{-(k+1)-1}},
\end{eqnarray*}
as desired.
\end{proof}
We will use the following freely:
\begin{lem}
We can always assume $f(d_kn_k)\leq f(n_k)\cdot{2^\varepsilon}$. More specifically, given $\varepsilon >0$ we can choose $d_k,n_k$ in such a way.
\end{lem}
\begin{proof}
Using Lemma \ref{lower bound f} we see that: $$f(n_k)\geq 2^{\frac{n_k}{2d_1\cdots d_{k-1}}}\gg d_k^2n_k^2\geq f(d_kn_k)-f(n_k).$$
Assuming $\{d_i,n_i\}_{i=0}^{k-1}$ and $d_k$ are fixed, we can let $n_k$ be large enough such that: $$(2^\varepsilon-1)f(n_k)\geq f(d_kn_k)-f(n_k),$$
and the claim follows.
\end{proof}
From now on, we take $n_k$ large enough so that $f(d_kn_k)\leq f(n_k)\cdot 2^{\frac{1}{3}}$.
\begin{prop}\label{f submul}
The function $f:\mathbb{N}\rightarrow \mathbb{N}$ constructed above is submultiplicative.
\end{prop}
\begin{proof}
We now turn to prove that $f$ is submultiplicative in the interval $[n_k+1,n_{k+1}]$ (by the induction hypothesis it is submultiplicative for $[1,n_k]$, where the induction base is Proposition \ref{f submul 0}). As in the basic step, we begin with the interval $[d_kn_k+1,n_{k+1}]$ (without limiting $n_{k+1}$, which can be thought of as infinity). Let $p+q>d_kn_k$ and as before, $p\leq q$. Denote $\beta=f(d_{k-1}n_{k-1})$. Then: $$f(p+q)\leq \beta\cdot 2^{\frac{n_k-d_{k-1}n_{k-1}}{2d_1\cdots d_{k-1}}+\frac{p+q-d_kn_k}{2d_1\cdots d_{k}}+\frac{1}{3}}=\beta\cdot 2^{\frac{p+q-d_kd_{k-1}n_{k-1}}{2d_1\cdots d_k}+\frac{1}{3}}.$$
We divide into cases:
\begin{itemize}
    \item \textbf{Suppose $p\leq d_{k-1}n_{k-1}$.} Note that for $t\in [n_k,d_kn_k-1]$ we have that $$\frac{f(t+1)}{f(t)}\leq 1+\frac{t+2}{f(t)}\leq 1+\frac{d_kn_k+1}{2^{\frac{n_k}{2d_1\cdots d_k}}}$$ which we can take to be smaller than $2^{\frac{1}{2d_1\cdots d_k}}$ by letting $n_k$ be large enough. Thus, if in addition we take $d_k\geq \frac{n_{k-1}}{2d_1\cdots d_{k-2}}$ then: $$\frac{f(p+q)}{f(q)}\leq 2^{\frac{p}{2d_1\cdots d_k}}\leq 2^{\frac{d_{k-1}n_{k-1}}{2d_1\cdots d_k}}\leq 2\leq f(p).$$
    (Note that the first inequality is evident if $q\geq d_kn_k$, and otherwise follows from the argument in the beginning of this case.)
    \item \textbf{If $d_{k-1}n_{k-1}<p\leq n_k$ then $q\geq n_k$ (as $d_k>2$), and assume in addition that $q\leq d_kn_k$.} 
    Note also that we can choose $d_k>d_{k-1}n_{k-1}+1$ so: $$d_{k-1}d_kn_{k-1}\leq (d_{k-1}n_{k-1}+1)(d_k-1)\leq p(d_k-1)$$ and thus (recalling that $q\leq d_kn_k$):
    \begin{eqnarray*}
    (\star)\ \ \ pd_k+d_kn_k-2d_{k-1}d_kn_{k-1}&\geq &  pd_k-p(d_k-1)-d_kd_{k-1}n_{k-1}+d_kn_k \\
    &\geq & p+q-d_{k-1}d_kn_{k-1}.\ \ \ \ \ 
    \end{eqnarray*}
    Then, using Condition (I): 
    \begin{eqnarray*}
    f(p)f(q)&\geq &f(p)f(n_k) \\
    &\geq &\beta^2\cdot 2^{\frac{p-d_{k-1}n_{k-1}}{2d_1\cdots d_{k-1}}-2^{-(k-1)-2}+\frac{n_k-d_{k-1}n_{k-1}}{2d_1\cdots d_{k-1}}-2^{-(k-1)-2}} \\
    &=&\beta^2\cdot 2^{\frac{pd_k+n_kd_k-2d_kd_{k-1}n_{k-1}}{2d_1\cdots d_{k}}-2^{-k}} \\
    &\geq & \beta\cdot 2^{\frac{p+q-d_kd_{k-1}n_{k-1}}{2d_1\cdots d_{k}}+\frac{1}{3}} \\
    &\geq &f(p+q).
    \end{eqnarray*}
    (The one before last inequality follows from $(\star)$ combined with the fact that $\beta \geq 2$.)
    \item \textbf{If $d_{k-1}n_{k-1}<p\leq n_k$ then $q\geq n_k$, and now assume that moreover $q>d_kn_k$.}
    Recalling Lemma \ref{lower bound f} we have: $$\frac{f(p+q)}{f(q)}\leq 2^{\frac{p}{2d_1\cdots d_k}}\leq f(p).$$

    \end{itemize}
        In the remaining cases,
        $p>n_k$. 
        \begin{itemize}
    \item \textbf{Assume $p\geq d_kn_k$.}
    Note that by Condition (I): 
    \begin{eqnarray*}
    f(p)f(q) & \geq & f(d_kn_k)^2\cdot 2^{\frac{p+q-2d_kn_k}{2d_1\cdots d_k}-2\cdot 2^{-k-2}},\\
    f(p+q) & \leq & f(d_kn_k)\cdot 2^{\frac{p+q-d_kn_k}{2d_1\cdots d_k}}.
    \end{eqnarray*}
    Therefore:
    \begin{eqnarray*}
    \frac{f(p)f(q)}{f(p+q)}& \geq & f(d_kn_k)\cdot 2^{-\frac{d_kn_k}{2d_1\cdots d_k}-2^{k-1}}\\
    &\geq & 2^{\frac{d_kn_k}{2d_1\cdots d_k}+1+2^{-k-1}}\cdot 2^{-\frac{d_kn_k}{2d_1\cdots d_k}-2^{k-1}}>1.
    \end{eqnarray*}
    (The middle inequality follows by Lemma \ref{lower bound f}.)
    
    
    
    \item \textbf{Suppose $n_k\leq p< d_kn_k$. If in addition $q\leq d_kn_k$} then $$f(p+q)\leq f(2d_kn_k)\leq \beta \cdot 2^{\frac{2d_kn_k-d_kd_{k-1}n_{k-1}}{2d_1\cdots d_k}+\frac{1}{3}}$$ and by Condition (I) specified for $x=n_k$:
    \begin{eqnarray*}
    f(p)f(q)\geq f(n_k)^2 & \geq & \beta^2\cdot 2^{\frac{2n_k-2d_{k-1}n_{k-1}}{2d_1\cdots d_{k-1}}-2\cdot 2^{-(k-1)-2}}\\
    &=& \beta^2\cdot 2^{\frac{2d_kn_k-2d_kd_{k-1}n_{k-1}}{2d_1\cdots d_k}-2^{-k}},
    \end{eqnarray*}
    and by Lemma \ref{lower bound f} applied for $x=d_{k-1}n_{k-1}$: $$\beta\geq 2^{\frac{d_{k-1}n_{k-1}}{2d_1\cdots d_{k-1}}+1}=2^{\frac{d_kd_{k-1}n_{k-1}}{2d_1\cdots d_{k}}+1}$$ so:
    $$f(p)f(q)\geq \beta\cdot 2^{\frac{2d_kn_k-d_kd_{k-1}n_{k-1}}{2d_1\cdots d_k}+1-2^{-k}}\geq f(p+q).$$
    \item \textbf{Suppose $n_k<p<d_kn_k<q$.} Then (applying Lemma \ref{lower bound f} on $x=p$, and Condition (I) on $x=q$: 
    \begin{eqnarray*}
    f(p)f(q)&\geq & 2^{\frac{p}{2d_1\cdots d_k}+1}f(q)\\
    &\geq & 2^{\frac{p}{2d_1\cdots d_k}+1}\cdot f(d_kn_k)\cdot 2^{\frac{q-d_kn_k}{2d_1\cdots d_k}-2^{-k-2}}\\
    &\geq & f(p+q).
    \end{eqnarray*}
\end{itemize}
We thus proved submultiplicativity of $f$ for the interval $[d_kn_k+1,n_{k+1}]$. It remains to show that $f(p+q)\leq f(p)f(q)$ for $p+q\in [n_k+1,d_kn_k]$.
\begin{itemize}
    \item \textbf{If $p>n_k$} then (applying Lemma \ref{lower bound f} for $x=n_k$):
    \begin{eqnarray*}
    f(p)f(q)\geq f(n_k)^2 & \geq & f(n_k)\cdot 2^{\frac{n_k}{2d_1\cdots d_k}+1}\\ & \geq & f(n_k)+d_k^2n_k^2\geq f(d_kn_k)\geq f(p+q),
    \end{eqnarray*}
    where the inequality $$f(n_k)\cdot 2^{\frac{n_k}{2d_1\cdots d_k}+1}\geq f(n_k)+d_k^2n_k^2$$ follows since if $d_k$ is fixed then the left hand side grows more rapidly than the right hand one (as a function of $n_k$), so in particular we can take $n_k$ large enough such that this inequality holds.
    \item \textbf{If $p\leq q<n_k$} then: 
    \begin{eqnarray*}
    f(p+q)& \leq & f(n_k)+(p+q-n_k)^2\\
    & \leq & f(n_k)+n_k^2\\
    & \leq & f(n_k)\cdot 2^{\frac{1}{2}}\\
    & \leq & \beta\cdot 2^{\frac{n_k-d_{k-1}n_{k-1}}{2d_1\cdots d_{k-1}}+\frac{1}{2}}.
    \end{eqnarray*}
    (The one before last inequality follows since $f$ grows exponentially in the interval $[d_{k-1}n_{k-1}+1,n_k]$, so in particular we can take $n_k$ to be large enough such that $f(n_k)\gg n_k^2$). 
    \begin{itemize}
\item  Suppose in addition that $d_{k-1}n_{k-1}<p$. Then, using Condition (I): 
    \begin{eqnarray*}
    f(p)f(q)&\geq & \beta ^2\cdot 2^{\frac{p+q-2d_{k-1}n_{k-1}}{2d_1\cdots d_{k-1}}-2\cdot 2^{-(k-1)-2}}\\
    &=& \beta^2\cdot 2^{\frac{p+q-d_{k-1}n_{k-1}}{2d_1\cdots d_{k-1}}}\cdot 2^{\frac{-d_{k-1}n_{k-1}}{2d_1\cdots d_{k-1}}-2^{-k}}.
    \end{eqnarray*}
    But applying Lemma \ref{lower bound f} for $x=d_{k-1}n_{k-1}$ we get $\beta\geq 2^{\frac{d_{k-1}n_{k-1}}{2d_1\cdots d_{k-1}}+1}$ so:
    \begin{eqnarray*}
    f(p)f(q) &\geq &\beta\cdot 2^{\frac{p+q-d_{k-1}n_{k-1}}{2d_1\cdots d_{k-1}}+1-2^{-k}} \\
    &\geq & \beta\cdot 2^{\frac{n_k-d_{k-1}n_{k-1}}{2d_1\cdots d_{k-1}}+1-2^{-k}} \\
    &\geq &\beta\cdot 2^{\frac{n_k-d_{k-1}n_{k-1}}{2d_1\cdots d_{k-1}}+\frac{1}{2}} \\
    &\geq & f(p+q).
    \end{eqnarray*}

    \item Now suppose $p\leq d_{k-1}n_{k-1}$ and $q$ is general. (recall $q\leq p+q\leq d_kn_k$). Note that if we make sure that $n_k>2d_{k-1}n_{k-1}$ then it is forced that $q>d_{k-1}n_{k-1}$. Now: $$\frac{f(p+q)}{f(q)}\leq 2^{\frac{p}{2d_1\cdots d_{k-1}}}\leq f(p)$$ where the last inequality follows from Lemma \ref{lower bound f} for $x=p$.
    \end{itemize}
    \item \textbf{It remains to take care of the case $d_{k-1}n_{k-1}\leq p\leq n_k\leq q$.} But notice that: $$f(p+q)\leq f(d_kn_k)\leq f(n_k)+d_k^2n_k^2\leq 2f(n_k)$$ since $f$ grows exponentially in the interval $[d_{k-1}n_{k-1}+1,n_k]$, so in particular we can take $n_k$ to be large enough such that $f(n_k)\gg d_k^2n_k^2$ (note $d_k$ is already fixed when we choose $n_k$). Now: $$f(p)f(q)\geq 2f(q)\geq 2f(n_k).$$
    \end{itemize}
We thus proved that $f:\mathbb{N}\rightarrow \mathbb{N}$ is a submultiplicative function.
\end{proof}

\section{Our construction is not equivalent to any growth function}

Let $f:\mathbb{N}\rightarrow \mathbb{N}$ be the increasing, submultiplicative function constructed in Section \ref{construction of f} with respect to the sequences $\{d_k,n_k\}_{k=1}^\infty$.

\begin{prop} \label{notequiv}
We can choose $\{d_i,n_i\}_{i=0}^{\infty}$ such that the resulting $f$ is not equivalent to any growth function of an finitely generated algebra.
\end{prop}
\begin{proof}
Since the conditions on the $d_k$ and $n_k$ in Section~\ref{construction of f} always require the parameters to be large enough (depending on those previously defined), we assume further that $n_k=km_k$ for $m_k$ to be determined in the sequel.
Suppose $C\in \mathbb{N}$ is given, and let $k=C$. Let $n=m_k+1$ and $D=\lfloor d_k(1-\frac{1}{m_k+1})\rfloor$, and observe that $\frac{1}{2}d_k\leq D\leq d_k$. We are going to contradict the property stated in Proposition \ref{property_growth} with these parameters, namely, we must show that: $$f(2CDn)-f(2CDn-C)>2D^2n(f(CDn)-f(Cn-C))^{2D}.$$ 
$2CDn\leq 2kd_km_k=2d_kn_k$ and as long as we take $n_k>2k$ we have that
$$2CDn-C\geq 2CDn-CD\geq 2kDm_k\geq kd_km_k=d_kn_k$$
so: 
\begin{eqnarray*}
f(2CDn)-f(2CDn-C)& \geq & f(2CDn-C)\cdot (2^{\frac{k}{2d_1\cdots d_k}-2^{-k-2}}-1)\\ & \geq & f(2CDn-C)\cdot \Delta_k
\end{eqnarray*}
where $\Delta_k$ is a value that does not depend on $n_k$ (we are going to take $n_k$ large enough to overcome this factor, which is very small).
Now, substituting our parameters and using Lemma \ref{lower bound f} for $x=d_kn_k$: $$f(2CDn-C)\geq f(d_kn_k)\geq 2^{\frac{d_kn_k}{2d_1\dots d_k}}=q^{n_k}=q^{km_k}$$ where $q=2^{\frac{1}{2d_1\cdots d_{k-1}}}$.
On the other hand, $$n_k=km_k=Cn-C<CDn\leq k\left(d_k\frac{m_k}{m_k+1}\right)(m_k+1)=d_kn_k$$ so (by definition of $f$ in intervals of this type): $$f(CDn)-f(Cn-C)\leq (CDn)^2\leq (kd_k(m_k+1))^2$$ hence:
\begin{eqnarray*}
2D^2n(f(CDn)-f(Cn-C))^{2D}& \leq & 2d_k^2(m_k+1)(kd_k(m_k+1))^{2d_k}\\
&\leq & (m_k+1)^{2d_k+1}\cdot \Gamma_k
\end{eqnarray*}
where $\Gamma_k$ depends only on $k,d_k$ (but not on $m_k$ or equivalently on $n_k$).
Finally, note that as we fix $k,d_1,\dots,d_k$ and let $m_k\rightarrow \infty$ we get:
\begin{eqnarray*}
f(2CDn)-f(2CDn-C)&\geq &(q^k)^{m_k}\cdot \Delta_k \\
&\gg & (m_k+1)^{2d_k+1}\cdot \Gamma_k \\
&\geq & 2D^2n(f(CDn)-f(Cn-C))^{2D}
\end{eqnarray*}
contradicting the property of Proposition \ref{property_growth}. Since this can be done for any $C=~k\in \mathbb{N}$, and we can take $d_k$ (and hence $D$) to be arbitrarily large, we proved that $f$ cannot be equivalent to any growth function of an finitely generated algebra.
\end{proof}

\section{Proof of Theorem \ref{main thm submul}}

\begin{prop} \label{superpoly}
Let $g:\mathbb{N}\rightarrow \mathbb{N}$ be an arbitrary subexponential function. Then we can choose $\{d_i,n_i\}_{i=0}^{\infty}$ such that the resulting function in our construction is $f$ satisfies $f\succeq g$.
\end{prop}
\begin{proof}

Since $g$ is assumed to be subexponential, there exists $\omega:\mathbb{N}\rightarrow \mathbb{R}$ such that $\omega(n)\xrightarrow{n\rightarrow \infty} 0$ and $g(n)\leq 2^{n\omega(n)}$.

We make sure that $n_k>\max\{m\,|\,\omega(m)\geq \frac{1}{2d_1\cdots d_k}\}$ for all $k\geq 1$.
We claim that for all $x\geq n_1$ we have that $f(x)\geq 2^{x\omega(x)}$. There are two possibilities: either $x\in [n_j,d_jn_j]$ or $x\in [d_jn_j,n_{j+1}]$ for some $j\geq 1$. Let us consider the first case. Then $x\geq n_j$ so by the way we picked $n_j$ we have that $\omega(x)<\frac{1}{2d_1\cdots d_j}$.
By Lemma \ref{lower bound f} we have that: $$f(x)\geq 2^{\frac{x}{2d_1\cdots d_j}}\geq 2^{x\omega(x)}.$$
As for the second case, if $x\in [d_jn_j,n_{j+1}]$ then (again using Lemma \ref{lower bound f} for $d_jn_j$): 
\begin{eqnarray*}
f(x) & \geq & 2^{\frac{x-d_jn_j}{2d_1\cdots d_j}-2^{-j-2}}f(d_jn_j)\\
& \geq & 2^{\frac{x}{2d_1\cdots d_j}+1-2^{-j-2}}\\
& \geq & 2^{x\omega(x)}.
\end{eqnarray*}
Thus $f(x)\geq 2^{x\omega(x)}\geq g(x)$ for all $x\gg 1$.
\end{proof}


Finally we have:

\begin{proof}[{Proof of Theorem \ref{main thm submul}}]
The theorem follows since we can take $\{d_k,n_k\}_{k=1}^\infty$ satisfying all conditions required in Propositions \ref{f submul}, \ref{notequiv} and \ref{superpoly}.
\end{proof}


\begin{thebibliography}{99}

\bibitem{BartholdiErschler}
L.~Bartholdi, A.~Erschler, \textit{Groups of given intermediate word growth}, Ann.~Inst.~Fourier (Grenoble) 64~5, 2003–-2036 (2014).

\bibitem{BellZelmanov}
J.~P.~Bell, E.~Zelmanov, \textit{On the growth of algebras, semigroups, and hereditary languages}, arXiv:1907.01777 [math.RA].

\bibitem{BBL1997}
A.~Ya.~Belov, V.~V.~BorisenkoV. N.~Latyshev, \textit{Monomial Algebras}, Journal of Mathematical Sciences 87 (3) (1997), 3463–-3575.

\bibitem{BergmanGap}
G.~M.~Bergman, \textit{A note on growth functions of Algebras and Semigroups},
mimeographed notes, University of California, Berkeley, 1978.

\bibitem{BGJalg}
B.~Greenfeld, \textit{Growth of monomial algebras, simple rings and free subalgebras}, J.~Algebra 489 (2017), 427--434.

\bibitem{BGIsr}
B.~Greenfeld, \textit{Prime and primitive algebras with prescribed growth types}, Israel Journal of Mathematics 220 (2017), 161--174.

\bibitem{KrauseLenagan}
G.~Krause, T.~Lenagan, \textit{Growth of algebras and the Gelfand-Kirillov dimension} (revised edition), Graduate studies in mathematics vol.~22, AMS Providence, Rhode Island (2000).

\bibitem{SmoktunowiczBartholdi}
A.~Smoktunowicz, L.~Bartholdi, \textit{Images of Golod-Shafarevich Algebras with Small Growth}, Quarterly Journal of Mathematics (2014) 65 (2): 421--438.

\bibitem{Semigroup}
V.~I.~Trofimov, \textit{The Growth Functions of Finitely Generated Semigroups}, Semigroup Forum (1980) 21 (4), 351--360.

\end{thebibliography}
\end{document}